\documentclass[11pt]{amsart}
\usepackage{amsmath,amssymb,amsthm,mathtools}
\usepackage[margin=1in]{geometry}

\usepackage{color,comment,appendix}

\newtheorem{theorem}{Theorem}[section]
\newtheorem{lemma}[theorem]{Lemma}
\newtheorem{proposition}[theorem]{Proposition}
\newtheorem{corollary}[theorem]{Corollary}
\theoremstyle{definition}
\newtheorem{definition}[theorem]{Definition}
\newtheorem{remark}[theorem]{Remark}
\newtheorem{example}[theorem]{Example}
\newtheorem{question}[theorem]{Question}
\newtheorem{conjecture}[theorem]{Conjecture}

\newcommand{\GL}{\mathrm{GL}}  
\newcommand{\PGL}{\mathrm{PGL}}  
\newcommand{\PSL}{\mathrm{PSL}}

\newcommand{\Hom}{\mathrm{Hom}}

\newcommand{\PP}{\mathbb{P}}

\usepackage[colorlinks=true, linkcolor=blue, anchorcolor=blue, citecolor=blue, filecolor=blue, menucolor= blue, urlcolor=blue,pdfencoding=auto, psdextra, pagebackref=true]{hyperref}

\begin{document}

\title[Finite subgroups from configurations of skew $n$-planes]{Collinearly complete sets and finite subgroups from configurations of  skew $n$-planes in $\PP^{2n+1}_K$}
\author{Giuseppe Favacchio}
\address[G. Favacchio]{Dipartimento di Ingegneria, Universit\`{a} degli studi di Palermo, Viale delle Scienze, `
90128 Palermo, Italy}
\email{giuseppe.favacchio@unipa.it}
\author{Jake Kettinger}
\address[J. Kettinger]{Department of Mathematics, Colorado State University, Fort Collins, Colorado, 80523, United States}
\email{jake.kettinger@colostate.edu}


\date{}

\begin{abstract}
We study collinearly complete finite sets of points arising from configurations
of pairwise skew $n$-planes in $\PP^{2n+1}_K$. To such a configuration we
associate a groupoid generated by the natural collinearity correspondences
between the $n$-planes, and we investigate the geometry of its finite orbits.

In characteristic zero, we prove a rigidity result for the case in which the
associated group is finite cyclic. After a suitable normalization, the matrices
defining the configuration are simultaneously diagonalizable and admit a common
two-block decomposition. Consequently, the orbit of a general point meets each
$n$-plane in a collinear set, and the full orbit is contained in a distinguished
projective $3$-space. This reduces the geometry of such orbits to the classical
case of skew lines in $\PP^3_{\mathbb C}$: inside the distinguished
$\PP^3_{\mathbb C}$, the orbit is geproci. We also prove that finite unions of
general orbits are cut out set-theoretically from the union of the $n$-planes
by a reducible surface.

Finally, we show that this characteristic-zero rigidity fails in positive
characteristic. In characteristic $2$, we construct cyclic examples whose orbit
slices are Fano plane configurations, and we exhibit a genuinely
higher-dimensional finite non-cyclic example in projective
$5$-space with associated group
$\PGL_3(\mathbb F_2)\cong \PSL_2(\mathbb F_7)$.
\end{abstract}
\maketitle

\section{Introduction}
Configurations of pairwise skew lines in $\PP^3_K$ give rise to a rich
interplay between projective geometry, finite group actions, and special point
sets. The connection between configuration groupoids, collinearly complete
sets, and geproci configurations in $\PP^3_K$ was introduced in
\cite{politus3}. For recent developments on configurations of skew lines and their associated projective subgroups, see \cite{politus7}, which classifies these sets of lines having a given finite group.
  For more on geproci sets, see also \cite{politus1,Kettinger2024,kettinger2025}.

The higher-dimensional analogue for configurations of pairwise skew
$n$-planes in $\PP^{2n+1}$, together with the corresponding finiteness
questions, was initiated by Ganger in her thesis \cite{ganger2024}. In this
paper we investigate this higher-dimensional setting using the groupoid
approach developed in \cite{f2025}. We focus on the structure of the associated group in
finite cases. 

The first part of the paper concerns characteristic $0$. Our main result in
this direction is Theorem~\ref{thm:finite-cyclic-orbit-line}. It shows that,
if $G_{\mathcal L}$ is finite cyclic, then the orbit of a general point is
highly constrained: the orbit slice on each $n$-plane is collinear, and the
lines obtained in this way fit together inside a distinguished projective
$3$-space containing the whole orbit. The key linear-algebra input is the
existence of a common two-block decomposition for the matrices defining the
configuration, obtained in Lemma~\ref{lem:block-compatibility}.

This reduction to a distinguished $\PP^3_{\mathbb C}$ allows us to return to the
classical geometry of skew lines. In Proposition~\ref{prop:P3-geproci}, we show
that each general orbit is geproci inside its distinguished $\PP^3_{\mathbb C}$. More
precisely, the intersections of the distinguished $\PP^3_{\mathbb C}$ with the
$n$-planes of the configuration form a configuration of skew lines, and the
orbit is collinearly complete with respect to this induced configuration.

We then prove a set-theoretic intersection result for finite unions of general
orbits. If $Z$ is a finite union of general orbits and $X$ denotes the union
of the $n$-planes in the configuration, then
Theorem~\ref{thm:union-orbits-set-theoretic-intersection} constructs a
reducible surface $Y\subset \PP^{2n+1}_{\mathbb C}$ such that $Z=X\cap Y$
set-theoretically. This result is weaker than the geprofi property asked for in
Question~\ref{conj:geprofi}, since geprofi would require a proper intersection
after a general projection. Nevertheless, it gives a higher-dimensional
analogue of the intersection phenomena arising from skew lines in $\PP^3_{\mathbb C}$.

We also discuss finite non-cyclic groups in characteristic $0$. We do not
know any example with $n>1$. The spectral obstructions developed in the
subsection \ref{sec:non-cyclic0}
suggest that the finite case may be much more rigid than in $\PP^3_{\mathbb C}$. This
leads to Conjecture~\ref{conj:finite-implies-cyclic}, which asks whether, in
characteristic $0$, every finite group $G_{\mathcal L}$ associated with a
configuration of $n$-planes in $\PP^{2n+1}_{\mathbb C}$, with $n>1$, must be cyclic.

The final part of the paper shows that the characteristic zero picture does
not extend to positive characteristic. In Section~\ref{sec:charp},
Example~\ref{ex:char2-2} constructs a cyclic example in characteristic $2$
with associated group $G_{\mathcal L}\cong C_7$. Unlike in characteristic
$0$, the orbit slices on the individual planes are not collinear; rather, the
seven points on each plane form a Fano plane configuration.
Proposition~\ref{prop:positive-char-singer} generalizes this construction
using Singer cycles: over $\mathbb F_q$, one obtains cyclic groups of order
$q^2+q+1$ whose general orbit slices span the whole plane.

Finally, Example~\ref{ex:char2} gives a genuinely non-cyclic finite example in
positive characteristic. In characteristic $2$, we construct a configuration
in $\PP^5_K$ whose associated group is
$G_{\mathcal L}\cong \PGL_3(\mathbb F_2)\cong \PSL_2(\mathbb F_7)$, of order
$168$. Proposition~\ref{prop:not-pgl2-char2} shows that this group does not
embed in $\PGL_2(K)$, emphasizing that the example is genuinely
higher-dimensional and not inherited from the classical line-configuration
case.

The paper is organized as follows. In Section~\ref{sec:prel} we introduce the
basic setup and the configuration groupoid associated with a finite collection
of pairwise skew $n$-planes in $\PP^{2n+1}_K$. In particular,
Section~\ref{sec:cc} develops the notion of collinearly complete sets in this
higher-dimensional setting. Section~\ref{sec:char0} is devoted to
characteristic $0$: we prove the finite cyclic rigidity theorem, construct
the distinguished $\PP^3_{\mathbb C}$, prove the geproci result inside it, and discuss
obstructions to non-cyclic finite groups. Finally, Section~\ref{sec:charp}
turns to positive characteristic, where we exhibit cyclic and non-cyclic
examples showing that the characteristic-zero rigidity phenomena do not
persist.

\section{Preliminaries and basic setup}\label{sec:prel}
Let $K$ be a field of arbitrary characteristic, and let $\PP^n_K$ denote projective space over $K$.
In this section we recall and extend some of the definitions and results from \cite{politus3} and \cite{f2025}.
Both papers are concerned with $\PP^3$; however, many of their constructions and arguments adapt naturally to $\PP^{2n+1}_K$.

In this section we describe normalized configurations in matrix form, compute the projectivities $f_{ijk}$ and the associated group $G_{\mathcal L},$ and then discuss examples and a geometric interpretation in terms of invariant subspaces. Related higher-dimensional groupoid constructions for pairwise skew $n$-subspaces in $\PP^{2n+1}_K$ were studied by Ganger; see in particular \cite[Chapters 5 and 6]{ganger2024}.

\subsection{The group of the groupoid}\label{sec:setup}
Let $\mathcal L=\{L_1,\dots,L_r\}$ be a finite set, with $r\ge 3$, of pairwise skew $n$-planes in $\PP^{2n+1}_K$.
For any triple of distinct indices $i,j,k$, we define a map
\[
f_{ijk}:L_i\longrightarrow L_j
\]
as follows: for $p\in L_i$, let
\[
T=\langle p,L_k\rangle
\]
be the $(n+1)$-plane spanned by $p$ and $L_k$, and set
\[
f_{ijk}(p):=T\cap L_j.
\]
Since the $n$-planes in $\mathcal L$ are pairwise skew, this intersection consists of a single point, so $f_{ijk}$ is well defined. Let $C_{\mathcal L}$ be the groupoid whose objects are the $n$-planes in $\mathcal L$ and whose morphisms are generated by the maps $f_{ijk}$, with composition whenever defined.
Since $r\ge 3$, for any pair of distinct indices $i, j$, there exists an index $k \neq i, j$. Thus, the map $f_{ijk}: L_i \to L_j$ is always a well-defined morphism in $C_{\mathcal L}$; in particular, $C_{\mathcal L}$ is connected.
\begin{definition}[Group of the groupoid {\cite[\S2]{politus3}}]\label{def:group-of-groupoid}
For $L_i\in\mathcal L$, the endomorphism set
\[
G_i(\mathcal L)=\Hom_{C_{\mathcal L}}(L_i,L_i)
\]
is a group under composition.
Since $C_{\mathcal L}$ is connected, the groups $G_i(\mathcal L)$ are all noncanonically isomorphic.
We denote their common isomorphism type by
$G_{\mathcal L},$
and call it the \emph{group of the groupoid} associated with $\mathcal L$.
\end{definition}

\subsection{Collinearly complete sets in \texorpdfstring{$\PP^{2n+1}$}{P{2n+1}}}\label{sec:cc}
The key geometric ingredient is the existence of a unique transversal line
through any point of one plane meeting two others. We establish this first,
then use it to define and characterize collinearly complete sets.

\begin{lemma}\label{lem:unique-line}
Let $L_i,L_j,L_k\subset \PP^{2n+1}_K$ be pairwise skew $n$-planes, and let $p\in L_i$.
Then there exists a unique line
\[
\ell=\ell_{jk}(p)\subset \PP^{2n+1}_K
\]
through $p$ meeting both $L_j$ and $L_k$.
Moreover,
\[
\ell\cap L_j=f_{ijk}(p)
\qquad\text{and}\qquad
\ell\cap L_k=f_{ikj}(p).
\]
\end{lemma}

\begin{proof}
Choose a hyperplane $H\subset \PP^{2n+1}_K$ such that $p\notin H$, and let 
$\pi_p:\PP^{2n+1}_K\dashrightarrow H$ 
be the projection from $p$ onto $H$. The images $\pi_p(L_j)$ and
$\pi_p(L_k)$ are two $n$-planes in $H\simeq \PP^{2n}_K$. Since
$n+n-2n=0$, they meet in a unique point, which we denote by $q_H$.

Let 
$\ell=\langle p,q_H\rangle.$ 
Then $\ell$ meets both $L_j$ and $L_k$, because $q_H$ lies in both
$\pi_p(L_j)$ and $\pi_p(L_k)$.

We now prove uniqueness. If $\ell'$ is any line through $p$ meeting both
$L_j$ and $L_k$, then $\pi_p(\ell')$ is a point of
$\pi_p(L_j)\cap \pi_p(L_k).$
By the uniqueness of $q_H$, we have $\pi_p(\ell')=q_H$, and hence
$\ell'=\langle p,q_H\rangle=\ell$.

Finally, the point $\ell\cap L_j$ is the unique point of $L_j$ lying in
$\langle p,L_k\rangle$, so $\ell\cap L_j=f_{ijk}(p)$. Similarly,
$\ell\cap L_k=f_{ikj}(p)$.
\end{proof}
The following notion is the natural higher-dimensional analogue of the definition in
\cite[Def.~2.1.10]{politus3} for the case $n=1$. This definition is also considered in Chapter~6 in \cite{ganger2024}.

\begin{definition}[Collinearly complete sets]\label{def:lc}
Let $\mathcal L$ be a finite collection of pairwise skew $n$-planes in $\PP^{2n+1}_K$, and let
$Z\subseteq \bigcup_{L\in\mathcal L}L$
be a nonempty set of points.
We say that $Z$ is \emph{collinearly complete with respect to $\mathcal L$} if whenever
$\ell\subset \PP^{2n+1}_K$
is a line meeting three distinct members $L_i,L_j,L_k\in\mathcal L$ and
$\ell\cap L_i\in Z,$
then
$\ell\cap L_j,\ \ell\cap L_k\in Z.$
\end{definition}

By Lemma~\ref{lem:unique-line} the following result is immediate.
\begin{proposition}\label{prop:lc-orbits}
A subset $Z\subseteq \bigcup_{L\in\mathcal L}L$ is collinearly complete if and only if it is closed under
the maps $f_{ijk}$.
Equivalently, $Z$ is collinearly complete if and only if it is a union of $C_{\mathcal L}$-orbits.
\end{proposition}

\begin{lemma}\label{lem:slices-projectively-equivalent}
Let $\mathcal L=\{L_1,\ldots,L_s\}$ be a finite collection, with $s\ge 3$, of pairwise skew $n$-planes in $\PP^{2n+1}_K$, and let
$Z\subseteq \bigcup_{L\in\mathcal L}L$
be a collinearly complete set with respect to $\mathcal L$.
Then for any two members $L_i,L_j\in\mathcal L$, the sets
\[
Z\cap L_i \subset L_i
\qquad\text{and}\qquad
Z\cap L_j \subset L_j
\]
are projectively equivalent.
More precisely, if $\varphi\in \Hom_{C_{\mathcal L}}(L_i,L_j)$, then
\[
\varphi(Z\cap L_i)=Z\cap L_j.
\]
\end{lemma}
\begin{proof}
Since $Z$ is collinearly complete, it is closed under the generators $f_{abc}$, hence under every morphism in
$C_{\mathcal L}$. Therefore, if $\varphi\in \Hom_{C_{\mathcal L}}(L_i,L_j)$, then
$\varphi(Z\cap L_i)\subseteq Z\cap L_j$.

Since $C_{\mathcal L}$ is a groupoid, $\varphi$ is invertible and
$\varphi^{-1}\in \Hom_{C_{\mathcal L}}(L_j,L_i)$. Applying the same argument to $\varphi^{-1}$ gives
$\varphi^{-1}(Z\cap L_j)\subseteq Z\cap L_i$, hence $Z\cap L_j\subseteq \varphi(Z\cap L_i)$.

Thus $\varphi(Z\cap L_i)=Z\cap L_j$. Since every morphism in $C_{\mathcal L}$ is a projective isomorphism between
the corresponding $n$-planes, the two slices are projectively equivalent.
\end{proof}

For $n=1$, collinearly complete sets are characterized by \cite[Theorem~3.0.3]{politus3}. This result shows that finite collinearly complete sets with respect to $\mathcal L$
are exactly the finite geproci half-grids, and conversely. That is,  $Z\subseteq \PP^3$ is a collinearly complete set with respect to $\mathcal L$ if and only if its general projection to $\PP^2$ is a complete intersection of two curves one of which can be chosen to be the projection of the lines in $\mathcal L$.

\begin{question}\label{q:cc-projection}
What geometric property of general projections replaces the half-grid / geproci   picture for finite collinearly complete sets in
$\PP^{2n+1}_K$ when $n>1$?
\end{question}
We return to this question in Question~\ref{conj:geprofi}, after introducing the necessary algebraic framework.

\subsection{A matrix representation of the group of the groupoid}\label{sec:matrix and groupoid}
To make the group $G_{\mathcal L}$ computable, we now represent the
$n$-planes by matrices and express the maps $f_{ijk}$ in terms of matrix
arithmetic. 
Following the idea in \cite{f2025}, in order to describe the generators of the group of the groupoid, it is useful to identify the $n$-planes with matrices. We recall here this approach.  We identify $\PP^{2n+1}_K$ with $\PP(K^{2n+2})$, with homogeneous coordinates
$[x_0:\dots:x_n:y_0:\dots:y_n]$.
Write $V_0=K^{n+1}\oplus 0$ and $V_\infty=0\oplus K^{n+1}$, so that
$K^{2n+2}=V_0\oplus V_\infty$.
We fix the two $n$-planes $L_0=\PP(V_0)$ and $L_\infty=\PP(V_\infty)$ in $\PP(K^{2n+2})$.

There is a one-to-one correspondence between the matrices in $\GL_{n+1}(K)$ and the $n$-planes in $\PP^{2n+1}_K$ skew to both $L_0$ and $L_\infty$.
For $M\in \GL_{n+1}(K)$, let \[\Gamma(M)=\{(u,Mu):u\in K^{n+1}\}\subseteq V_0\oplus V_\infty,\] and set
$L_M=\PP(\Gamma(M))$.
We also fix $L_1=L_I$, the $n$-plane corresponding to the identity matrix.

The following criterion is the natural higher-dimensional analogue of the matrix description used in the line case in \cite[Lemma~2.3]{f2025}. Since the proof is the same linear-algebraic argument, we include it only for the reader's convenience.

\begin{lemma}
	\label{lem:skew-criterion}
	Let $M,N\in \GL_{n+1}(K)$. Then
	\[
	L_M\cap L_N=\emptyset
	\qquad\Longleftrightarrow\qquad
	M-N\in \GL_{n+1}(K).
	\]
\end{lemma}

\begin{proof}
	A point of $L_M\cap L_N$ is represented by a nonzero vector $(u,Mu)=(v,Nv)$
	in $V_0\oplus V_\infty$. Since the first components agree, we must have $u=v$, and hence $(M-N)u=0.$
	Therefore $L_M\cap L_N\neq \emptyset$ if and only if there exists a nonzero vector
	$u\in K^{n+1}$ such that $(M-N)u=0$, namely if and only if $M-N$ is singular.
\end{proof}

Assume from now on that
\[
\mathcal L=\{L_\infty,L_0,L_1,L_2,\dots,L_r\},
\]
where $L_1=L_I$ and, for $i\ge 2$, one has $L_i=L_{M_i}$ with $M_i\in \GL_{n+1}(K)$. We call such a configuration of pairwise skew $n$-planes a {\em normalized configuration.}

For notational convenience, set $M_0=0$, so that $L_0$ is the plane $\PP(V_0)$, while $L_\infty=\PP(V_\infty)$ remains distinguished and is not represented by a matrix.

Note that the next results exclude the index $\infty$. Indeed, the group $G_{\mathcal L}$ is generated by the actions with finite indices. The formulas involving $\infty$ will only be used in the proof of Corollary \ref{cor:generators-by-differences}.
\begin{lemma}[Cf. {\cite[Lemma 2.5]{f2025}}] \label{lem:formula-fijk}
Let $i,j,k\in \{0,\ldots,r \}$ be distinct indices.
Then, the map
$f_{ijk}:L_i\to L_j$ corresponds to the class
\[
[(M_j-M_k)^{-1}(M_i-M_k)]\in \PGL_{n+1}(K).
\]
\end{lemma}
\begin{proof}
     The proof  is identical to {\cite[Lemma 2.5]{f2025}}, replacing lines by $n$-spaces.
\end{proof}
\begin{theorem}\label{thm:generators}
With notation as above, the group $G_{\mathcal L}\subseteq \PGL_{n+1}(K)$ is generated by the projective classes
\[
[(M_j-M_k)^{-1}(M_i-M_k)]
\]
for distinct indices $i,j,k\in \{0,\ldots,r \}$.
\end{theorem}
\begin{proof}
The proof is identical to that of \cite[Corollary~2.7]{f2025}, using Lemma~\ref{lem:formula-fijk}. 
\end{proof}

\begin{corollary}\label{cor:generators-by-differences}
In the normalized setting, the group $G_{\mathcal L}$ is generated by the projective classes of the differences
\[
[M_i-M_j]\in \PGL_{n+1}(K),
\qquad i\neq j\in \{0,\ldots, r\}.
\]
\end{corollary}
\begin{proof}
Although Theorem~\ref{thm:generators} shows that $G_{\mathcal L}$ is generated by the actions with finite indices, we may also consider the map $f_{i\infty k}.$ Exactly as in \cite[Theorem~2.6]{f2025}, its class is $[M_i-M_k].$ Hence the classes $[M_i-M_j]$ belong to $G_{\mathcal L},$ and they generate $G_{\mathcal L}.$
\end{proof}

\begin{remark}\label{rem:torsion-finite}
The analogue of \cite[Proposition~2.1.3(b)]{politus3} holds in the present setting as well:
the group $G_{\mathcal L}$ is finite if and only if every element of $G_{\mathcal L}$ has finite order, see \cite[Proposition 6.1]{ganger2024}. In particular, the group $G_{\mathcal L}$ is finitely generated.
\end{remark}

We now illustrate the previous results with two examples showing different behaviours of finite collinearly complete sets.

 In higher dimension, finite orbits of infinite groups need not lie on transversal lines (cf. \cite[Theorem~2.1.6]{politus3}); however, in the examples considered here they are still supported on special transversal linear subspaces.

\begin{example}\label{ex:infinite-group-finite-orbit}
Let $K=\mathbb Q(d)$, where $d$ is a primitive fifth root of unity, so that $d^5=1$ and
$d\neq 1$. Consider the normalized configuration
$\mathcal L=\{L_\infty,L_0,L_1,L_2\}\subset \PP^5_K,
$
where $L_1=L_I$ and $L_2=L_{M_2}$ with
\[
M_2=
\begin{pmatrix}
-d&0&0\\
0&-d^4&1\\
0&0&-d^4
\end{pmatrix}.
\]

Since $\det(M_2)\neq 0$, the plane $L_2$ is skew to both $L_0$ and $L_\infty$.
Moreover,
\[
M_2-I=
\begin{pmatrix}
-(1+d)&0&0\\
0&-(1+d^4)&1\\
0&0&-(1+d^4)
\end{pmatrix},
\]
and this matrix is also invertible, because $1+d\neq 0$ and $1+d^4\neq 0$.
Hence $L_2$ is also skew to $L_1=L_I$, so the four planes in $\mathcal L$ are pairwise skew.

By Corollary~\ref{cor:generators-by-differences}, the group $G_{\mathcal L}$ is generated by the classes
$[M_2]$ and $[M_2-I]$ in $\PGL_3(K)$.
We first show that $G_{\mathcal L}$ is infinite. Replacing $M_2$ by a scalar multiple, we may write
\[
[M_2]=
\begin{bmatrix}
d^2&0&0\\
0&1&-d\\
0&0&1
\end{bmatrix}.
\]
For every $m\ge 1$ one has
\[
[M_2^m]=
\begin{bmatrix}
d^{2m}&0&0\\
0&1&-md\\
0&0&1
\end{bmatrix}.
\]
Thus the classes $[M_2^m]$ are all distinct, so $[M_2]$ has infinite order.
Therefore $G_{\mathcal L}$ is infinite.

Now consider the point $v_0=[1:1:0]\in \PP^2_K$ and
\[
p_0=[0,v_0]=[0:0:0:1:1:0]\in L_\infty.
\]

For $[u:v:0]\in \PP^2_K$, the matrix $M_2$ acts by
\[
[u:v:0]\longmapsto [d^2u:v:0].
\]
Similarly, scaling $M_2-I$ by $-(1+d^4)^{-1}$ does not change its projective class, and gives
\[
[M_2-I]=
\begin{pmatrix}
\frac{1+d}{1+d^4}&0&0\\
0&1&-\frac{1}{1+d^4}\\
0&0&1
\end{pmatrix}.
\]
Since $(1+d)/(1+d^4)=d$, the class $[M_2-I]$ acts by
\[
[u:v:0]\longmapsto [du:v:0].
\]
Hence the subgroup generated by $[M_2]$ and $[M_2-I]$ acts on the points $[u:v:0]$ through the maps
\[
[u:v:0]\longmapsto [d^m u:v:0],\qquad m\in \mathbb Z.
\]
In particular,
\[
[p_0]_{\mathcal L}\cap L_\infty
=[p_0]_{G_{\mathcal L}}
=
\{[d^m:1:0]:m\in\mathbb Z\}
=
\{[1:1:0],[d:1:0],[d^2:1:0],[d^3:1:0],[d^4:1:0]\},
\]
which has cardinality $5$, and $|[p_0]_{\mathcal L}|=20$.

The set $[p_0]_{\mathcal L}$ is degenerate; indeed, it is contained in the projective 3-space defined by
$x_2=y_2=0,$ 
which meets each of the four planes in $\mathcal L$ in a line.
\end{example}
\begin{example}\label{ex:12-points-P5}
Let $K=\mathbb Q$. Consider the normalized configuration
$\mathcal L=\{L_\infty,L_0,L_1,L_2\}\subset \PP^5_K$,
where $L_1=L_I$ and $L_2=L_{M_2}$ with
\[
M_2=
\begin{pmatrix}
2&0&0\\
0&3&0\\
0&0&5
\end{pmatrix}.
\]

Since $\det(M_2)\neq 0$, the plane $L_2$ is skew to both $L_0$ and $L_\infty$.
Moreover,
$M_2-I$ 
is also invertible, so $L_2$ is skew to $L_1$ as well.

The common transversals to $L_\infty,L_0,L_1,L_2$ correspond to the fixed points of the projectivity induced by $M_2$ on $L_\infty.$ Since $M_2$ is diagonal with three distinct eigenvalues, it has exactly three eigenlines, namely $v_1=[1:0:0],$ $v_2=[0:1:0],$ and $v_3=[0:0:1].$ Hence there are exactly three lines meeting all four planes.

Let
\[
p_i=[0,v_i]\in L_\infty,\qquad i=1,2,3,
\]
and set
\[
Z=[p_1]_{\mathcal L}\cup [p_2]_{\mathcal L}\cup [p_3]_{\mathcal L}.
\]
Then $Z$ consists of $12$ points and has $h$-vector
\[
h_Z=(1,5,3,3).
\]
In particular, $Z$ is nondegenerate in $\PP^5$.

Moreover, the general projection of $Z$ to $\PP^4$ has $h$-vector
$(1,4,4,3).$

For each of the three transversal lines $\ell_1,\ell_2,\ell_3$, choose a general plane
$\Pi_1,\Pi_2,\Pi_3\subset \PP^5$ containing it.
Then $Z$ is contained in
\[
\Bigl(\bigcup_{i=0}^2 L_i\cup L_\infty\Bigr)\cap \Bigl(\Pi_1\cup \Pi_2\cup \Pi_3\Bigr).
\]
After a sufficiently general projection $\pi:\PP^5\dashrightarrow \PP^4$, the image $\pi(Z)$ is therefore contained in
\[
\Bigl(\pi(L_\infty)\cup \pi(L_0)\cup \pi(L_1)\cup \pi(L_2)\Bigr)\cap
\Bigl(\pi(\Pi_1)\cup \pi(\Pi_2)\cup \pi(\Pi_3)\Bigr),
\]
which has no common positive-dimensional component in $\PP^4$, that is, the intersection of a union of four planes with a union of three planes in $\PP^4$ is proper and has degree $4\cdot 3=12.$
Since  $ |\pi(Z)|=12$ it follows that $\pi(Z)$ coincides with this intersection. 
\end{example}

The two examples above illustrate a common feature: the orbit is supported
on a linear subspace that meets each plane of $\mathcal L$ in equal
dimension. The following lemma characterizes exactly when such a subspace
exists.

\begin{lemma}\label{lem:transversal-subspace}
Let $\mathcal L=\{L_\infty,L_0,L_1,L_M\}\subset \PP^{2n+1}_K$ be a normalized configuration of $n$-planes, for some $M\in \GL_{n+1}(K)$.
Fix an integer $d$ with $1\le d\le n$.
Then the following are equivalent:

\begin{enumerate}
\item there exists a projective subspace $\Sigma\subset \PP^{2n+1}_K$ of dimension $2d-1$ intersecting each $n$-plane in $\mathcal L$ in a linear space of projective dimension $d-1$;

\item there exists a $d$-dimensional subspace $E\le K^{n+1}$ such that $M(E)\subseteq E$.
\end{enumerate}

In this case, there is a unique $d$-dimensional subspace $E\le K^{n+1}$ such that $\Sigma=\PP(E\oplus E).$ \end{lemma}
\begin{proof}
Assume first that there exists a projective subspace $\Sigma=\PP(W)\subset \PP(K^{n+1}\oplus K^{n+1})$
of dimension $2d-1$ such that each of $\Sigma\cap L_0$ and $\Sigma\cap L_\infty$ has dimension $d-1$.
Then $\dim W=2d$, while
$\dim(W\cap (K^{n+1}\oplus 0))=d$ and $\dim(W\cap (0\oplus K^{n+1}))=d$.

Since these two subspaces intersect trivially and their dimensions sum
to $2d = \dim W$, their direct sum equals $W$.
 
  Hence
\[W=(E_0\oplus 0)\oplus (0\oplus E_\infty)\]
for some $d$-dimensional subspaces $E_0,E_\infty\le K^{n+1}.$

Now $\Sigma\cap L_1$ has dimension $d-1$, so $\dim(W\cap \Gamma(I))=d$.
Since $\Gamma(I)=\{(u,u):u\in K^{n+1}\}$, one has
$W\cap \Gamma(I)=\{(u,u):u\in E_0\cap E_\infty\}$.
Therefore $\dim(E_0\cap E_\infty)=d$, hence $E_0=E_\infty$.
Write this common subspace as $E$.
Then $W=E\oplus E$.

Finally, $\Sigma\cap L_M$ has dimension $d-1$, so $\dim((E\oplus E)\cap \Gamma(M))=d$.
But $(E\oplus E)\cap \Gamma(M)=\{(u,Mu):u\in E,\ Mu\in E\}$, so this has dimension $d$
if and only if $Mu\in E$ for every $u\in E$, that is, if and only if $M(E)\subseteq E$.

Conversely, suppose that there exists a $d$-dimensional subspace $E\le K^{n+1}$ with $M(E)\subseteq E$,
and let $\Sigma=\PP(E\oplus E)$.
Then $\dim(E\oplus E)=2d$, so $\Sigma$ has dimension $2d-1$.
Then $\Sigma \cap L_0 = \PP(E \oplus 0)$, $\Sigma \cap L_\infty = \PP(0 \oplus E)$,  $\Sigma\cap L_1=\PP\{(u,u):u\in E\},$ and 
$\Sigma\cap L_M=\PP\{(u,Mu):u\in E\}.$

Since $M(E)\subseteq E$, all these intersections are projective subspaces of dimension $d-1$. This proves~(1).
\end{proof}

Thus invariant subspaces of the matrices defining the configuration
naturally produce auxiliary linear spaces meeting the planes of
$\mathcal L$ in equal dimension.

The examples above suggest that finite collinearly complete sets may
arise as intersections of two varieties after a suitable projection.
This motivates the following definition, introduced in \cite{politus4}.

\begin{definition}[Special case of Definition 1.2 in \cite{politus4}]\label{def:geprofi} Let $Z$ be a finite set of points in $\PP^{2n+1}_K$. We say that $Z$ is $(a,b)$-geprofi if a general projection of $Z$ to $\PP^{2n}_K$ is the proper intersection of two n-dimensional varieties of degrees $a$ and $b.$
\end{definition}

\begin{question}\label{conj:geprofi}
Let $\mathcal L$ be a set of $r$ pairwise skew $n$-planes in $\PP^{2n+1}_K$ such that
$G_{\mathcal L}$ is finite, and let $Z$ be a finite union of $G_{\mathcal L}$-orbits.
Assume that $|Z\cap L|=s$ for every $L\in\mathcal L$.
Is $Z$ an $(s,r)$-geprofi set?
More precisely, if $Z$ is a finite collinearly complete set for $\mathcal L$, is it true that after a general projection the image of $Z$ can be realized as the intersection of the image of the union of the $n$-planes in $\mathcal L$ with a second variety of degree $s$?
\end{question}
In the next section we study this question in the case where $G_{\mathcal L}$
is finite cyclic. We prove that general orbits are contained in distinguished
projective $3$-spaces and are geproci inside them, obtaining a set-theoretic
intersection result for finite unions of general orbits. We also prove the
corresponding geprofi statement for $n\ge 3$. The case $n=2$ remains open in
general, although it holds in some special examples.

\section{The group of the groupoid in characteristic 0}\label{sec:char0}
In this section we study the case in which $G_{\mathcal L}$ is a finite cyclic group over $\mathbb C$.
This is the first finite case that can be analyzed completely.
Our main goal is to show that, although the configuration lives in
$\PP^{2n+1}_{\mathbb C}$, the finite cyclic case forces a strong linear degeneracy:
for a general point $p$ on one of the $n$-planes of $\mathcal L$, and for
every $L\in\mathcal L$, the set
$[p]_{\mathcal L}\cap L$
is contained in a line in $L$.
Equivalently, each slice of the orbit spans a projective subspace of dimension
at most one.
As a consequence, each general orbit is naturally contained in a distinguished
projective $3$-space, where the geometry reduces to the classical half-grid geproci picture in $\PP^3_{\mathbb C}$.
\subsection{Cyclic groups in characteristic 0} 
We begin with a basic finite cyclic example, which already illustrates the main phenomenon of this section: although the configuration lives in $\PP^{2n+1}_{\mathbb C}$, the orbit of a general point is forced to become collinear on each member of the configuration.

\begin{example}\label{ex:D4 in P^5}Let $\varepsilon$ be a third root of unity and consider 
\[M=\begin{pmatrix}
-\varepsilon^2&0&0\\
0&-\varepsilon^2&0\\
0&0&-\varepsilon\\
\end{pmatrix}
\in \mathrm{GL}_3(\mathbb C). \]

    The set $\mathcal L=\{L_0, L_{\infty}, L_I, L_M\}$ defines the group $G_{\mathcal L}=\langle [M]\rangle \cong C_3.$

Via the eigenspace decomposition, the group determines on each plane of $\mathcal L$ a distinguished fixed point and a distinguished fixed line. For example on $L_{\infty}$ it fixes the point $[0:0:1]$ and all the points on the line $\langle [1:0:0], [0:1:0]\rangle$. Moreover, note that the image of a point $[a:b:c]$ under the action of $[M]$ is $[a:b:\varepsilon c]$. Therefore the orbit of a point $p$ restricted to $L_{\infty}$ lies on the line $\langle p, [0:0:1] \rangle.$

Let $p_1,\ldots,p_r$ be general points on $L_\infty$, and write
\[
[p_1,\ldots,p_r]_{\mathcal L}=[p_1]_{\mathcal L}\cup\cdots\cup [p_r]_{\mathcal L}.
\]
For each $p_i$, there exists a projective $3$-space $\Lambda_i$ containing the orbit $[p_i]_{\mathcal L}$ and meeting
$L_\infty$ in the line $\langle p_i,[0:0:1]\rangle$.

For a general point $p_i\in L_\infty,$ the orbit $[p_i]_{\mathcal L}$ is a set of 12 points isomorphic to the $D_4$ configuration, see \cite[Chapter 3]{politus1}. Thus $[p_i]_{\mathcal L}$ is contained in a cubic surface in $\Lambda_i,$ namely a cone whose vertex is a general point of $\Lambda_i.$ 
Hence $[p_1,\ldots,p_r]_{\mathcal L}$ is the finite set-theoretic intersection of the union of the four planes in $\mathcal L$ with a nondegenerate surface of degree $3r$ in $\PP^5_{\mathbb C}.$
Therefore the general projection of $[p_1,\ldots,p_r]_{\mathcal L}$ is a proper intersection of the projection of these two varieties.

We computed the h-vectors of projection of the set $[p_1,\ldots,p_r]_{\mathcal L}$  for small $r$ using random points. 
\[
\begin{array}{l|l|l|l}
  Z   &  h_Z &  \pi: \PP^5 \dashrightarrow \PP^4& \pi: \PP^5 \dashrightarrow \PP^3\\
     \hline
 \  [p_1]_{\mathcal L}  &(1,3,6,\ 2) &&\\
 \  [p_1,p_2]_{\mathcal L}  &(1, 5, 11,\ 7) & (1, 4, 9,\ 10)&(1, 3, 6, 10,\ 4)\\
 \  [p_1,p_2,p_3]_{\mathcal L}  &(1, 5, 14, 12,\ 4)&(1, 4, 10, 16,\ 5)&(1, 3, 6, 10, 14,\ 2)\\
  \ [p_1,\ldots,p_4]_{\mathcal L}  &(1, 5, 14, 16,\ 8,\ 4) & (1, 4, 10, 18, 11,\ 4)&(1, 3, 6, 10, 14, 14)\\
\ [p_1,\ldots,p_5]_{\mathcal L}  &(1, 5, 14, 20, 12,\ 8) &(1, 4, 10, 19, 17,\ 9)&(1, 3, 6, 10, 14, 18,\ 8)\\
\end{array}
\]
Moreover, the general projection to $\PP^2$ seems to have generic Hilbert function for $r\ge 2$.
\end{example}

We begin with a general linear-algebra observation on finite cyclic projective actions. It will later be applied to the action induced on the $n$-planes of a normalized configuration.

\begin{lemma}\label{lem:general-cyclic-orbit-span}
Let $G$ be a finite cyclic subgroup of $\PGL_{n+1}(\mathbb C)$, and let $g$ be a generator.
Choose a lift $A\in \GL_{n+1}(\mathbb C)$ of $g$.
Let
\[
\mathbb C^{n+1}=E_1\oplus\cdots\oplus E_t
\]
be the decomposition into eigenspaces corresponding to the distinct eigenvalues of $A$.

Then, for a general point $p=[v]\in \PP^n_{\mathbb C}$, writing $v=v_1+\cdots+v_t$ with
$v_j\in E_j$ and $v_j\neq 0$ for every $j$, the orbit $G\cdot p$ spans the projective subspace
\[
\PP(\langle v_1,\dots,v_t\rangle)\simeq \PP^{t-1}_{\mathbb C}.
\]
In particular, a general orbit under $G$ spans a line if and only if the action has exactly two distinct eigenvalues.
\end{lemma}

\begin{proof}
Since $[A]$ has finite order in $\PGL_{n+1}(\mathbb C)$, there exist
$s\geq 1$ and $\lambda\in\mathbb C^*$ such that
\[
A^s=\lambda I.
\]
Hence the minimal polynomial of $A$ divides $x^s-\lambda$, which has
distinct roots over $\mathbb C$. Therefore $A$ is diagonalizable.

For every integer $m\ge 0$, the point $g^m(p)$ is represented by
\[
A^m v=\lambda_1^m v_1+\cdots+\lambda_t^m v_t,
\]
where $\lambda_1,\dots,\lambda_t$ are the distinct eigenvalues of $A$.
Hence the orbit is contained in the projectivization of the vector space
$\langle v_1,\dots,v_t\rangle$.

Conversely, let $v=v_1+\cdots+v_t$, where $v_i$ is a nonzero vector in the eigenspace corresponding to $\lambda_i$ for each $i=1,\dots,t.$ Since the vectors $v_1,\dots,v_t$ belong to distinct eigenspaces, they are linearly independent. Relative to the basis $v_1,\dots,v_t,$ the vectors $A^m v$ for $m=0,\dots,t-1$ have coordinates $(\lambda_1^m,\dots,\lambda_t^m),$ hence they form a Vandermonde matrix in the eigenvalues $\lambda_1,\dots,\lambda_t.$ Since these eigenvalues are pairwise distinct, this matrix is invertible, and therefore the vectors $A^m v$ span the same vector space as $v_1,\dots,v_t$. Thus the orbit of $[v]$ spans the projective subspace $\PP(\langle v_1,\dots,v_t\rangle)$, which has dimension $t-1$.
\end{proof}

\begin{lemma}\label{lem:simultaneously-diagonalizable}Let $M_1,\dots,M_r\in \GL_{n+1}(\mathbb C)$, and assume that the projective classes
\[[M_1],\dots,[M_r]\in \PGL_{n+1}(\mathbb C)\] all belong to a finite cyclic subgroup.
Then the matrices $M_1,\dots,M_r$ are simultaneously diagonalizable over $\mathbb C$.
\end{lemma}
\begin{proof}
Let $\langle [A]\rangle$ be the finite cyclic subgroup containing all the classes $[M_i]$.
Then for each $i$ there exist $a_i\in \mathbb Z$ and $c_i\in \mathbb C^*$ such that
$M_i=c_iA^{a_i}$.

Since $[A]$ has finite order in $\PGL_{n+1}(\mathbb C)$, there exist $s\ge 1$ and
$\lambda\in \mathbb C^*$ such that $A^s=\lambda I$.
Choose $\mu\in \mathbb C^*$ with $\mu^s=\lambda$, and set $B=\mu^{-1}A$.
Then $B^s=I$, so $B$ is diagonalizable over $\mathbb C$.

Now $A=\mu B$, hence $M_i=c_i\mu^{a_i}B^{a_i}$ for every $i$.
Thus each $M_i$ is a scalar multiple of a power of the diagonalizable matrix $B$, and therefore
the matrices $M_1,\dots,M_r$ are simultaneously diagonalizable.
\end{proof}

The next lemma shows that, in the normalized $4$-plane case, finiteness
of the cyclic group forces the defining matrix to have at most two distinct
eigenvalues.

\begin{lemma}\label{lem:two-eigenvalues}
Let $\mathcal L=\{L_\infty,L_0,L_1,L_M\}\subset \PP^{2n+1}_{\mathbb C}$ be a normalized configuration, and assume that
$G_{\mathcal L}$ is a finite cyclic group.
Then $M$ has at most two distinct eigenvalues.
\end{lemma}

\begin{proof}
Since $G_{\mathcal L}$ is finite cyclic, both $[M]$ and $[M-I]$ have finite order in
$\PGL_{n+1}(\mathbb C)$.
Let $\lambda_1,\dots,\lambda_{n+1}$ be the eigenvalues of $M$.
Because $[M]$ has finite order, all ratios $\lambda_i/\lambda_j$ are roots of unity, hence
$|\lambda_i|=|\lambda_j|$ for all $i,j$. Thus the eigenvalues of $M$ all lie on a circle centered at $0$.
Likewise, since $[M-I]$ has finite order, all ratios $(\lambda_i-1)/(\lambda_j-1)$ are roots of unity, hence
$|\lambda_i-1|=|\lambda_j-1|$ for all $i,j$. Thus the eigenvalues also lie on a circle centered at $1$.

These two circles are distinct, since they have different centers, and therefore they meet in at most two points. Hence the set $\{\lambda_1,\dots,\lambda_{n+1}\}$ contains at most two distinct elements.
\end{proof}

The following lemma extends this eigenvalue constraint to the full
configuration: not only does each matrix have at most two distinct
eigenvalues, but all matrices share a common two-block decomposition.

\begin{lemma}\label{lem:block-compatibility}
Assume that
\[
\mathcal L=\{L_\infty,L_0,L_1,L_2,\dots,L_r\}
\]
is a normalized configuration of pairwise skew $n$-planes in
$\PP^{2n+1}_{\mathbb C}$, and assume that $G_{\mathcal L}$ is finite
cyclic. Suppose moreover that, after a change of coordinates, the matrices
$M_2,\dots,M_r$ are diagonal.
Write
\[
M_j=\operatorname{diag}(m_{j1},\dots,m_{j,n+1})
\qquad\text{for } j=2,\dots,r.
\]
Then, for every $j\geq 2$, the diagonal entries of $M_j$ take at most two distinct values.

Moreover,  there is a partition of $\{1,\ldots,n+1\}=U\sqcup W$ such that the matrix $M_j$ is constant on each block; that is,
\[
m_{ju}=m_{ju'}\quad\text{for all }u,u'\in U,
\]
and
\[
m_{jv}=m_{jv'}\quad\text{for all }v,v'\in W.
\]
\end{lemma}
\begin{proof}
For every $j\geq 2$, the classes
\[
[M_j]=[M_j-M_0]
\qquad\text{and}\qquad
[M_j-I]=[M_j-M_1]
\]
belong to $G_{\mathcal L}$. Since $G_{\mathcal L}$ is finite cyclic, both
classes have finite order. Therefore Lemma~\ref{lem:two-eigenvalues} applies
to $M_j$, and $M_j$ has at most two distinct eigenvalues. Since $M_j$ is
diagonal, its diagonal entries take at most two distinct values.

We now prove that the corresponding two-block decompositions are compatible.
Let $M_i$ and $M_j$ be two non-scalar matrices among the $M_2,\dots,M_r$.
Suppose, by contradiction, that their partitions of the coordinate indices are
not the same. Then there exist indices $u,v,w$ such that
\[
m_{iu}=m_{iv}\ne m_{iw}
\qquad\text{and}\qquad
m_{ju}\ne m_{jv}=m_{jw}.
\]
Set
\[
\alpha=m_{iu}=m_{iv},\qquad
\delta=m_{iw},
\]
and
\[
\beta=m_{ju},\qquad
\gamma=m_{jv}=m_{jw}.
\]
Thus $\alpha\ne\delta$ and $\beta\ne\gamma$.

By the proof of Lemma~\ref{lem:two-eigenvalues}, applied to $M_i$, the two
numbers $\alpha$ and $\delta$ have the same distance from $0$ and also the
same distance from $1$, so they are symmetric with respect to the real axis. Since they are distinct, they must be complex
conjugates:
\[
\delta=\overline{\alpha}.
\]
Similarly,
$\gamma=\overline{\beta}.$ 
In particular, $\alpha$ and $\beta$ are non-real.

Now consider the diagonal matrix $M_i-M_j$. Since the planes $L_i$ and
$L_j$ are skew, $M_i-M_j$ is invertible. Moreover, the projective class 
$[M_i-M_j]$ 
belongs to $G_{\mathcal L}$. Hence it has finite order in
$\PGL_{n+1}(\mathbb C)$. Therefore all ratios of its diagonal entries are
roots of unity, and in particular all its diagonal entries have the same
modulus.

Looking at the diagonal entries of $M_i-M_j$ in positions $u,v,w$, we get
\[
|\alpha-\beta|
=
|\alpha-\gamma|
=
|\delta-\gamma|.
\]
Using $\delta=\overline{\alpha}$ and $\gamma=\overline{\beta}$, the first
equality gives
\[
|\alpha-\beta|=|\alpha-\overline{\beta}|.
\]
Write
\[
\alpha=a+bi,\qquad \beta=c+di.
\]
Since $\alpha$ and $\beta$ are non-real, we have $b\ne 0$ and $d\ne 0$.
But
\[
|\alpha-\beta|^2=(a-c)^2+(b-d)^2
\]
whereas
\[
|\alpha-\overline{\beta}|^2=(a-c)^2+(b+d)^2.
\]
The equality
\[
|\alpha-\beta|=|\alpha-\overline{\beta}|
\]
therefore implies
\[
(b-d)^2=(b+d)^2,
\]
hence $bd=0$, a contradiction.

Thus the partitions associated to any two non-scalar matrices $M_i$ and
$M_j$ coincide. Choosing this common partition as
\[
\{1,\ldots,n+1\}=U\sqcup W,
\]
we conclude that every non-scalar $M_j$ is constant on $U$ and on $W$.
The scalar matrices are constant on every subset, so the same partition works
for all $M_j$.
\end{proof}

The spectral constraints established above now translate directly into a
geometric statement: on each $n$-plane of the configuration, the orbit
of a general point is collinear.

\begin{corollary}\label{cor:finite-cyclic-orbit-line}
Let $\mathcal L=\{L_\infty,L_0,L_1,\dots,L_r\}\subset \PP^{2n+1}_{\mathbb C}$ be a normalized configuration, and assume that
$G_{\mathcal L}$ is a finite cyclic group.

Then there exists a decomposition $\mathbb C^{n+1}=U\oplus W$ with
$1\le \dim U,\dim W\le n$ such that, after a suitable choice of basis, every matrix $M_i$ has the form
\[
M_i=\operatorname{diag}(\underbrace{a_i,\dots,a_i}_{\dim U},
\underbrace{b_i,\dots,b_i}_{\dim W})
\]
for suitable $a_i,b_i\in\mathbb C^*$.
Hence, for a general point $p$ on any member of $\mathcal L$, the orbit $[p]_{\mathcal L}$ on that $n$-plane is contained in a line.
\end{corollary}
\begin{proof}
By Lemma~\ref{lem:simultaneously-diagonalizable}, the matrices defining the
normalized configuration may be simultaneously diagonalized. Thus, after a
common change of basis, we may write
\[
M_i=\operatorname{diag}(m_{i1},\dots,m_{i,n+1})
\qquad\text{for }i=1,\dots,r.
\]

By Lemma~\ref{lem:block-compatibility}, there exists a partition
\[
\{1,\dots,n+1\}=U\sqcup W
\]
such that every matrix $M_i$ is constant on each block. Equivalently, after a
possible reordering of the coordinates, there is a decomposition
\[
\mathbb C^{n+1}=U\oplus W
\]
with respect to which every $M_i$ has the form
\[
M_i=
\operatorname{diag}
\left(
\underbrace{a_i,\dots,a_i}_{\dim U},
\underbrace{b_i,\dots,b_i}_{\dim W}
\right)
\]
for suitable $a_i,b_i\in\mathbb C^*$.

If all the $M_i$ are scalar, then one may choose any nontrivial decomposition
\[
\mathbb C^{n+1}=U\oplus W
\]
with $1\leq \dim U,\dim W\leq n$. Otherwise, the partition above is induced
by a non-scalar matrix, and therefore both blocks are nonempty. Hence in all
cases we may assume
\[
1\leq \dim U,\dim W\leq n.
\]

We now prove the geometric statement. Let $p$ be a general point on one
member of $\mathcal L$.

By Corollary~\ref{cor:generators-by-differences}, $G_{\mathcal L}$ is
generated by the classes $[M_i-M_j]$. Since each $M_i-M_j$ acts by the
scalar $a_i-a_j$ on $U$ and by $b_i-b_j$ on $W$, and since this property
is preserved under composition and inversion, every element of
$G_{\mathcal L}$ acts by a scalar on $U$ and by a scalar on $W$.
Therefore, a representative of each point of $[p]_{\mathcal L}\cap L$
has the form $\lambda u+\mu w$ with $u\in U$, $w\in W$, $\lambda,\mu\in\mathbb C^*$.

For a general point $p$, both components $u$ and $w$ are nonzero. Hence
all these points lie in the projective line
$\PP(\langle u,w\rangle).$
Equivalently, by Lemma~\ref{lem:general-cyclic-orbit-span}, the relevant
finite cyclic action has at most two eigenspaces, so the span of a general
orbit has dimension at most one.

Therefore, for every $L\in\mathcal L$, the set
$[p]_{\mathcal L}\cap L$
is contained in a line.
\end{proof}

The following theorem shows that, in the finite cyclic case, the orbit of a
general point is collinear on each $n$-plane of the configuration, and
moreover these lines fit together into a single projective $3$-space
containing the whole orbit.

\begin{theorem}\label{thm:finite-cyclic-orbit-line}
Let $\mathcal L=\{L_\infty,L_0,L_1,\ldots,L_r\}\subset \PP^{2n+1}_{\mathbb C}$
be a normalized configuration of pairwise skew $n$-planes, and assume that
$G_{\mathcal L}\subseteq \PGL_{n+1}(\mathbb C)$ is a finite nontrivial cyclic group.
Then, for a general point $p\in L_\infty$, the set
$[p]_{\mathcal L}\cap L_\infty$
consists of $|G_{\mathcal L}|$ collinear points.
By projective equivalence, the same holds on every member of $\mathcal L$.

Moreover, for every general point $p\in L_\infty$, there exists a projective $3$-space
$\Lambda_p\subset \PP^{2n+1}_{\mathbb C}$ such that, for every $L\in\mathcal L$, the set
$[p]_{\mathcal L}\cap L$ spans the line $\Lambda_p\cap L$.
In particular, for general $p$, the orbit $[p]_{\mathcal L}$ spans $\Lambda_p\simeq \PP^3$.
\end{theorem}

\begin{proof}
By Lemma~\ref{lem:simultaneously-diagonalizable}, the matrices $M_2,\dots,M_r$
defining the planes of $\mathcal L$ are simultaneously diagonalizable over
$\mathbb C$. By Lemma~\ref{lem:block-compatibility}, after a suitable choice
of basis there exists a decomposition
\[
\mathbb C^{n+1}=U\oplus W
\]
with $1\leq \dim U,\dim W\leq n$ such that every matrix $M_i$ has the form
\[
M_i=
\operatorname{diag}
\left(
\underbrace{a_i,\dots,a_i}_{\dim U},
\underbrace{b_i,\dots,b_i}_{\dim W}
\right).
\]
Thus each $M_i$ acts by the scalar $a_i$ on $U$ and by the scalar $b_i$
on $W$. 

Since each difference $M_i-M_j$ is scalar on $U$ and on $W$, and since
this property is preserved under composition and inversion, Corollary~\ref{cor:generators-by-differences}
implies that every element of $G_{\mathcal L}$ acts by a scalar on $U$ and
by a scalar on $W$.

\medskip
\noindent\textit{Collinearity of the orbit on each plane.}
Let $p\in L_\infty$ be a general point, and write its corresponding vector as
$v=u+w, \qquad u\in U,\quad w\in W,$
with $u,w\neq 0$. Every element of $G_{\mathcal L}$ sends $[v]$ to a
point of the form
$[\lambda u+\mu w]$
for some $\lambda,\mu\in\mathbb C^*$. Hence the slice
$[p]_{\mathcal L}\cap L_\infty$
is contained in the projective line
$\PP(\langle u,w\rangle).$

Since $G_{\mathcal L}$ is nontrivial, a general orbit contains at least two
distinct points, and therefore this slice spans that line. By
Lemma~\ref{lem:slices-projectively-equivalent}, the same holds on every member
of $\mathcal L$.

\medskip
\noindent\textit{The distinguished $\PP^3$.}
Set
\[
E_p=\langle u,w\rangle\subset \mathbb C^{n+1}.
\]
The subspace $E_p$ is invariant under every $M_i$, because
\[
M_i(u)=a_i u\in \langle u\rangle
\qquad\text{and}\qquad
M_i(w)=b_i w\in \langle w\rangle.
\]
Therefore, by the same argument as in
Lemma~\ref{lem:transversal-subspace}, the projective subspace
\[
\Lambda_p=\PP(E_p\oplus E_p)\subset \PP^{2n+1}_{\mathbb C}
\]
is a projective $3$-space. It meets $L_\infty$ and $L_0$ in the lines
\[
\PP(E_p\oplus 0)
\qquad\text{and}\qquad
\PP(0\oplus E_p),
\]
and, for $i=1,\dots,r$, it meets
$L_i=L_{M_i}$
in the line
$\PP\{(x,M_i x):x\in E_p\}.$

We claim that
\[
[p]_{\mathcal L}\cap L\subseteq \Lambda_p\cap L
\]
for every $L\in\mathcal L$. For $L=L_i$, the points of the slice are
represented by vectors of the form
\[
(x,M_i x),
\qquad
x\in E_p.
\]
Since $M_i(E_p)\subseteq E_p$, these vectors lie in $E_p\oplus E_p$, and
therefore the corresponding points lie in $\Lambda_p\cap L_i$. The cases
$L=L_\infty$ and $L=L_0$ are immediate from the descriptions above.

Since the slice on $L_\infty$ spans the line
$\Lambda_p\cap L_\infty,$
and the same holds on $L_0$ by projective equivalence, the full orbit spans
both lines
$\Lambda_p\cap L_\infty$
and
$\Lambda_p\cap L_0.$
These two lines are skew in $\Lambda_p\simeq \PP^3_{\mathbb C}$, so they span
$\Lambda_p$. Hence the full orbit $[p]_{\mathcal L}$ spans
$\Lambda_p\simeq \PP^3_{\mathbb C}.$

Finally, for every $L\in\mathcal L$, the containment
$\langle [p]_{\mathcal L}\cap L\rangle\subseteq \Lambda_p\cap L$
was proved above, while the slice $[p]_{\mathcal L}\cap L$ spans a line by
projective equivalence. Hence
$\langle [p]_{\mathcal L}\cap L\rangle=\Lambda_p\cap L.$
\end{proof}

Since each general orbit is contained in a distinguished $\PP^3_{\mathbb C}$ and is
collinearly complete with respect to the induced configuration of skew
lines, the theory from \cite{politus3} applies.

\begin{proposition}\label{prop:P3-geproci}
With the notation of Theorem~\ref{thm:finite-cyclic-orbit-line}, let $p\in L_\infty$ be a general point, and let
$\Lambda_p\simeq \PP^3_{\mathbb C}$ be the projective $3$-space spanned by the orbit $[p]_{\mathcal L}$.
Then the finite set
$[p]_{\mathcal L}\subset \Lambda_p\simeq \PP^3_{\mathbb C}$
is geproci.
\end{proposition}
\begin{proof}
By Theorem~\ref{thm:finite-cyclic-orbit-line}, for every $L\in\mathcal L$ the
set $[p]_{\mathcal L}\cap L$ spans the line $\Lambda_p\cap L$.
Since the $n$-planes of $\mathcal L$ are pairwise skew in $\PP^{2n+1}_{\mathbb C}$,
the lines $\{\Lambda_p\cap L\}_{L\in\mathcal L}$ are pairwise skew in
$\Lambda_p\simeq\PP^3_{\mathbb C}$ (a common point of $\Lambda_p\cap L_i$ and
$\Lambda_p\cap L_j$ would lie in $L_i\cap L_j=\varnothing$).

It remains to show that $[p]_{\mathcal L}$ is collinearly complete with
respect to the configuration of skew lines
\[
\{\Lambda_p\cap L\}_{L\in\mathcal L}
\]
inside $\Lambda_p\simeq\PP^3_{\mathbb C}$. Let $\ell\subset\Lambda_p$ be a line meeting
three distinct members
\[
\Lambda_p\cap L_i,\qquad \Lambda_p\cap L_j,\qquad \Lambda_p\cap L_k
\]
of this configuration, and suppose that
\[
q:=\ell\cap(\Lambda_p\cap L_i)\in [p]_{\mathcal L}.
\]
Since $\ell\subset\Lambda_p\subset\PP^{2n+1}_{\mathbb C}$, the line $\ell$ is also a
line through $q\in L_i$ meeting both $L_j$ and $L_k$. By
Lemma~\ref{lem:unique-line}, it is the unique such line, and therefore
\[
\ell\cap L_j=f_{ijk}(q),
\qquad
\ell\cap L_k=f_{ikj}(q).
\]
Since $[p]_{\mathcal L}$ is an orbit under the groupoid
$C_{\mathcal L}$, it is closed under all maps $f_{ijk}$
by Proposition~\ref{prop:lc-orbits}. Therefore
\[
f_{ijk}(q),\ f_{ikj}(q)\in [p]_{\mathcal L}.
\]
This proves that $[p]_{\mathcal L}$ is collinearly complete with respect to
the configuration of skew lines
\[
\{\Lambda_p\cap L\}_{L\in\mathcal L}.
\]

Thus \cite[Theorem~3.0.3]{politus3} applies and shows that
$[p]_{\mathcal L}\subset\Lambda_p\simeq\PP^3_{\mathbb C}$ is geproci.
\end{proof}

The next results provide a weaker intersection-theoretic statement in the
finite cyclic case. Each general orbit is geproci inside its distinguished
$\PP^3_{\mathbb C}$, and finite unions of such orbits can be cut out set-theoretically
from the union of the $n$-planes of the configuration by a reducible surface.
This does not yet imply the geprofi property from
Question~\ref{conj:geprofi}.

\begin{theorem}\label{thm:union-orbits-set-theoretic-intersection}
Let
$\mathcal L=\{L_\infty,L_0,L_1,\ldots,L_r\}\subset \PP^{2n+1}_{\mathbb C}$
be a normalized configuration such that $G_{\mathcal L}$ is a finite cyclic
group. Let
$Z=O_1\cup\cdots\cup O_m$
be a finite union of orbits, and assume that each orbit $O_\nu$ is general
enough for Theorem~\ref{thm:finite-cyclic-orbit-line} to apply. Set
$X=\bigcup_{L\in\mathcal L}L.$
Then there exists a $2$-dimensional reducible subvariety
$Y\subset \PP^{2n+1}_{\mathbb C}$ such that $Z=X\cap Y$ set-theoretically.
\end{theorem}

\begin{proof}
Fix one orbit $O_\nu$. By Theorem~\ref{thm:finite-cyclic-orbit-line}, there
exists a projective $3$-space $\Lambda_\nu\subset \PP^{2n+1}_{\mathbb C}$
such that $O_\nu\subset \Lambda_\nu$ and, for every $L\in\mathcal L$, the
set $O_\nu\cap L$ spans the line $\Lambda_\nu\cap L$. By
Proposition~\ref{prop:P3-geproci}, the orbit $O_\nu\subset
\Lambda_\nu\simeq\PP^3_{\mathbb C}$ is geproci with respect to the configuration of
pairwise skew lines $\{\Lambda_\nu\cap L : L\in\mathcal L\}$.

Choose a general projection $\pi_\nu:\Lambda_\nu\dashrightarrow \PP^2_{\mathbb C}$ with
center $P_\nu\in\Lambda_\nu$. Since $O_\nu$ is geproci, there exists a
plane curve $D_\nu\subset \PP^2_{\mathbb C}$ such that
\[
\pi_\nu(O_\nu)
=
\pi_\nu\Bigl(\bigcup_{L\in\mathcal L}(\Lambda_\nu\cap L)\Bigr)\cap D_\nu
\]
as a proper intersection in $\PP^2_{\mathbb C}$. Let $Y_\nu = \pi_\nu^{-1}(D_\nu)
\subset \Lambda_\nu$ be the cone over $D_\nu$ with vertex $P_\nu$. Then
$Y_\nu$ is a surface contained in $\Lambda_\nu$.

\medskip
\noindent\textit{Claim: $Y_\nu\cap L = O_\nu\cap L$ for every
$L\in\mathcal L$.}

Set $\ell_{\nu,L}=\Lambda_\nu\cap L$. Since $Y_\nu\subset\Lambda_\nu$,
every point of $Y_\nu\cap L$ lies in $\Lambda_\nu\cap L = \ell_{\nu,L}$,
and conversely $Y_\nu\cap\ell_{\nu,L}\subseteq Y_\nu\cap L$, so
\[
Y_\nu\cap L = Y_\nu\cap\ell_{\nu,L}.
\]
Since $P_\nu$ is general it does not lie on $\ell_{\nu,L}$, so the
restriction $\pi_\nu|_{\ell_{\nu,L}}:\ell_{\nu,L}\to\pi_\nu(\ell_{\nu,L})$
is an isomorphism. Since $Y_\nu=\pi_\nu^{-1}(D_\nu)$ by definition,
\[
Y_\nu\cap\ell_{\nu,L}
= \pi_\nu^{-1}(D_\nu)\cap\ell_{\nu,L}
= (\pi_\nu|_{\ell_{\nu,L}})^{-1}
  \!\left(D_\nu\cap\pi_\nu(\ell_{\nu,L})\right).
\]
By the geproci property of $O_\nu$ and the definition of $D_\nu$,
\[
D_\nu\cap\pi_\nu(\ell_{\nu,L}) = \pi_\nu(O_\nu\cap\ell_{\nu,L}).
\]
Since $\pi_\nu|_{\ell_{\nu,L}}$ is injective,
\[
Y_\nu\cap\ell_{\nu,L} = O_\nu\cap\ell_{\nu,L},
\]
and therefore $Y_\nu\cap L = O_\nu\cap L$, as claimed.

\medskip
Now set $Y=\bigcup_{\nu=1}^m Y_\nu$. Each $Y_\nu$ is a surface, so $Y$
is a $2$-dimensional reducible subvariety of $\PP^{2n+1}_{\mathbb C}$.

We show $Z = X\cap Y$. For the inclusion $Z\subseteq X\cap Y$, let
$q\in O_\nu$. Since $O_\nu\subset X$, there is some $L\in\mathcal L$
such that $q\in O_\nu\cap L$. By the claim,
\[
O_\nu\cap L=Y_\nu\cap L,
\]
so $q\in Y_\nu\subseteq Y$. Hence $q\in X\cap Y$.

For the reverse inclusion, let $x\in X\cap Y$. Then $x\in L$ for some
$L\in\mathcal L$ and $x\in Y_\nu$ for some $\nu$, hence
$x\in L\cap Y_\nu$. By the claim, $L\cap Y_\nu = O_\nu\cap L\subseteq Z$.
Thus $X\cap Y\subseteq Z$, and therefore $X\cap Y = Z$ set-theoretically.
In particular, $X\cap Y$ is finite.
\end{proof}

\begin{proposition}\label{prop:geprofi}
Keep the notation and hypotheses of
Theorem~\ref{thm:union-orbits-set-theoretic-intersection}, and assume
moreover that $n\ge 3$.
Let $O_\nu$ be one of the orbits, let $s=|O_\nu\cap L|$ for any
$L\in\mathcal L$ (this is independent of $L$ by
Lemma~\ref{lem:slices-projectively-equivalent}), and let
$r=|\mathcal L|$.
Then $O_\nu$ is $(s,r)$-geprofi.
\end{proposition}

\begin{proof}
By Theorem~\ref{thm:union-orbits-set-theoretic-intersection}, we have
$Z=X\cap Y$ in $\PP^{2n+1}_{\mathbb C}$, where
$X=\bigcup_{L\in\mathcal L}L$ has dimension $n$, and
$Y=\bigcup_\mu Y_\mu$ with each $Y_\nu$ a surface of degree $s$.
Let $\rho:\PP^{2n+1}_{\mathbb C}\dashrightarrow\PP^{2n}_{\mathbb C}$ be a general projection
with center $P\in\PP^{2n+1}_{\mathbb C}$.
Since $Z$ is finite, $\rho$ is injective on $Z$ for a general center.

Fix one orbit $O_\nu$. We claim that
$\rho(O_\nu)=\rho(X)\cap\rho(Y_\nu)$.
The Claim in the proof of
Theorem~\ref{thm:union-orbits-set-theoretic-intersection} shows that
$Y_\nu\cap L=O_\nu\cap L$ for every $L\in\mathcal L$, hence
$X\cap Y_\nu=O_\nu$.
Therefore $\rho(O_\nu)\subseteq \rho(X)\cap\rho(Y_\nu)$.

For the reverse inclusion, suppose
$q\in \rho(X)\cap\rho(Y_\nu)\setminus\rho(O_\nu)$.
Then there exist $x\in X$ and $y\in Y_\nu$ with
$\rho(x)=\rho(y)=q$ and $x\neq y$, so the line $\overline{xy}$ passes
through $P$.
Moreover, $x\notin Y_\nu$: if $x\in Y_\nu$, then
$x\in X\cap Y_\nu=O_\nu$, hence $q=\rho(x)\in \rho(O_\nu)$,
a contradiction.
Thus $x\in X\setminus Y_\nu$ and $y\in Y_\nu$.

So every additional intersection point gives a line through $P$ joining a
point of $X\setminus Y_\nu$ to a point of $Y_\nu$.
All such lines are contained in the join $J(X\setminus Y_\nu,Y_\nu)$,
which has dimension at most
$\dim(X\setminus Y_\nu)+\dim Y_\nu+1\le n+3$.
Since $n\ge 3$, we have $n+3\le 2n<2n+1=\dim\PP^{2n+1}_{\mathbb C}$, so this join is
a proper subvariety of $\PP^{2n+1}_{\mathbb C}$.
Hence, for a general center $P$, no additional point exists, and
$\rho(O_\nu)=\rho(X)\cap\rho(Y_\nu)$.

This intersection is proper because
$\dim\rho(X)+\dim\rho(Y_\nu)=n+2\le 2n=\dim\PP^{2n}_{\mathbb C}$.
Also, $\deg\rho(X)=r$, since a general codimension-$n$ linear subspace of
$\PP^{2n}_{\mathbb C}$ pulls back to a general codimension-$n$ linear subspace of
$\PP^{2n+1}_{\mathbb C}$ through $P$, which meets each member of $\mathcal L$ in one
point.
Moreover, $\deg\rho(Y_\nu)=\deg Y_\nu=s$, since a general projection
preserves the degree of $Y_\nu$.
Thus $\rho(O_\nu)$ is the proper intersection of an $n$-dimensional
variety of degree $r$ and a surface of degree $s$, so $O_\nu$ is
$(s,r)$-geprofi.
\end{proof}

\begin{corollary}
With the same notation, if $Z=O_1\cup\cdots\cup O_m\subseteq \PP^{2n+1}_{\mathbb C}$, with $n>2, $ is a finite union of
pairwise distinct orbits, then for a general projection
$\rho:\PP^{2n+1}_{\mathbb C}\dashrightarrow\PP^{2n}_{\mathbb C}$ one has
$\rho(Z)=\rho(X)\cap\bigl(\rho(Y_1)\cup\cdots\cup\rho(Y_m)\bigr)$.
In particular, $\rho(Z)$ is the proper intersection of a degree-$r$
variety with a variety of degree $ms$.
\end{corollary}

\begin{proof}
Applying Proposition~\ref{prop:geprofi} to each orbit, we get
$\rho(O_\nu)=\rho(X)\cap\rho(Y_\nu)$ for all $\nu$.
Since the orbits are pairwise disjoint and $\rho$ is injective on the
finite set $Z=\bigcup_\nu O_\nu$, the sets $\rho(O_\nu)$ are pairwise
disjoint as well.
Taking the union over $\nu$ gives the claim.
\end{proof}

\begin{remark}
When $n=1$, the orbits are already geproci in $\PP^3_{\mathbb C}$ by
Proposition~\ref{prop:P3-geproci}, consistently with \cite{politus3}.
The argument above works for $n\ge 3$.
The obstruction in the case $n=2$ is only an obstruction to the present argument.
Indeed, the join $J(X\setminus Y_\nu,Y_\nu)$ may fill all of $\PP^5_{\mathbb C}$, so the secant-line dimension count no longer excludes extraneous intersections after projection.
However, this does not imply that such intersections must occur.
Thus the geprofi property remains open in general for $n=2$, although
it holds in Example~\ref{ex:12-points-P5} by a direct degree argument.
\end{remark}

\subsection{Towards non-cyclic finite groups in characteristic 0}\label{sec:non-cyclic0}
The results of Section~\ref{sec:char0} give a complete picture of the
finite cyclic case in characteristic $0$. We now turn to the question of
whether non-cyclic finite groups can arise.

At present we do not know any example of a non-cyclic finite group $G_{\mathcal L}$ when $n>1$.
The goal of this subsection is to collect obstructions which suggest that the finite case may be substantially more rigid than in $\PP^3_{\mathbb C}$.

The point is the following.
In the cyclic case, every normalized $4$-plane subconfiguration
\[
\{L_\infty,L_0,L_1,L_M\}
\]
forces the matrix $M$ to have at most two distinct eigenvalues.
If one wants to construct a finite non-cyclic group, one must therefore consider at least two matrices,
say $M$ and $N$, and require not only that $[M]$ and $[N]$ have finite order, but also that
$[M-I]$, $[N-I]$, and also the projective classes coming from the differences
\[
M-I,\qquad N-I,\qquad M-N
\]
must have finite order. Equivalently, after normalizing the mixed configuration, one is led to matrices
such as $(M-N)(I-N)^{-1}$, so the finite case imposes restrictions not only on
$M$ and $N$, but also on their interaction.

With our convention for the graphs $L_M$, the corresponding normalized matrix is
\[
X=(M-N)(I-N)^{-1}.
\]
Thus the finite case imposes simultaneous restrictions on $M$, $N$, and on the interaction between them.
In particular, if $M$ and $N$ have two-eigenvalue decompositions which are not compatible, then one expects
the matrix $X$, and already $M-N$ itself, to acquire three distinct eigenvalues.
This would contradict the necessary condition coming from the $4$-plane case.

For this reason, before attempting to construct non-cyclic finite groups, it is natural to study the spectrum of
$M-N$ when $M$ and $N$ each have exactly two distinct eigenvalues.

We recall a standard result.
\begin{lemma}\label{lem:two-eigenvalues-projector}
Let $M\in \GL_{n+1}(\mathbb C)$ be diagonalizable with exactly two distinct eigenvalues $a$ and $b$.
Let
\[
\mathbb C^{n+1}=E_a\oplus E_b
\]
be the corresponding eigenspace decomposition.
Then
\[
M=bI+(a-b)Q,
\]
where $Q\in \operatorname{End}(\mathbb C^{n+1})$ is the unique linear map such that
$Q(v)=v$ for every $v\in E_a$ and $Q(v)=0$ for every $v\in E_b$.
\end{lemma}

\begin{proof}
This is standard, but we include the argument for the reader's convenience.
The map $Q$ is well defined because every vector $v\in \mathbb C^{n+1}$ can be written uniquely as
$v=v_a+v_b$ with $v_a\in E_a$ and $v_b\in E_b$.
By definition, $Q(v)=v_a$.

The endomorphism $bI+(a-b)Q$ acts as multiplication by $a$ on $E_a$ and
as multiplication by $b$ on $E_b$. Hence it agrees with $M$ on the direct
sum $E_a\oplus E_b=\mathbb C^{n+1}$. Therefore $M=bI+(a-b)Q.$
\end{proof}

\begin{lemma}\label{lem:spectrum-difference-general}
Let $M=bI+\alpha Q$ and $N=dI+\beta P$ be matrices in
$\GL_{n+1}(\mathbb C)$, where $P,Q\in \operatorname{End}(\mathbb C^{n+1})$
satisfy $P^2=P$ and $Q^2=Q$. Then
\[
M-N=(b-d)I+(\alpha Q-\beta P).
\]
Moreover, every vector in $\ker(P)\cap\ker(Q)$ is an eigenvector of $M-N$
with eigenvalue $b-d$. Finally, if $\lambda\neq b-d$ is an eigenvalue of
$M-N$, then $\lambda-(b-d)$ is an eigenvalue of the restriction of
\[
\alpha Q-\beta P
\]
to the invariant subspace
\[
\operatorname{Im}(P)+\operatorname{Im}(Q).
\]
\end{lemma}
\begin{proof}
The identity is immediate.

Let $v\in\ker(P)\cap\ker(Q)$. Then $Pv=Qv=0$, hence
\[
(M-N)v=(b-d)v.
\]
Thus every vector in $\ker(P)\cap\ker(Q)$ is an eigenvector of $M-N$ with
eigenvalue $b-d$.

The subspace $\operatorname{Im}(P)+\operatorname{Im}(Q)$ is invariant under
$\alpha Q-\beta P$, since $P$ and $Q$ have images contained in this
subspace.

Now let $\lambda\neq b-d$ be an eigenvalue of $M-N$, and let $v\neq 0$
be a corresponding eigenvector. Then
\[
(\alpha Q-\beta P)v=(\lambda-(b-d))v.
\]
The left-hand side lies in $\operatorname{Im}(P)+\operatorname{Im}(Q)$, and
$\lambda-(b-d)\neq 0$. Hence
\[
v\in \operatorname{Im}(P)+\operatorname{Im}(Q).
\]
Therefore $\lambda-(b-d)$ is an eigenvalue of the restriction of
$\alpha Q-\beta P$ to $\operatorname{Im}(P)+\operatorname{Im}(Q)$.
\end{proof}

\begin{corollary}\label{cor:spectrum-difference}
Let $M=bI+\alpha Q$ and $N=dI+\beta P$ be matrices in $\GL_3(\mathbb C)$, where
$P,Q\in \operatorname{End}(\mathbb C^3)$ satisfy $P^2=P$, $Q^2=Q$, and
$\operatorname{rank}(P)=\operatorname{rank}(Q)=1$.
Then the eigenvalues of $M-N$ are
\[
b-d,\qquad b-d+\lambda_+,\qquad b-d+\lambda_-,
\]
where $\lambda_+,\lambda_-$ are the roots of 
$\lambda^2-(\alpha-\beta)\lambda+\alpha\beta(\operatorname{tr}(PQ)-1)=0.$

In particular, $M-N$ has at most two distinct eigenvalues if and only if either
\[
0\in\{\lambda_+,\lambda_-\}
\qquad \text{or}\qquad
\lambda_+=\lambda_-.
\]
Equivalently, this happens if and only if either
\[
\alpha\beta(\operatorname{tr}(PQ)-1)=0
\qquad \text{
or}
\qquad
(\alpha-\beta)^2-4\alpha\beta(\operatorname{tr}(PQ)-1)=0.
\]
\end{corollary}

\begin{proof}
Set $A=\alpha Q-\beta P$. By Lemma~\ref{lem:spectrum-difference-general},
every vector in $\ker(P)\cap\ker(Q)$ is an eigenvector of $M-N$ with
eigenvalue $b-d$. Since $P$ and $Q$ have rank $1$ on $\mathbb C^3$,
we have $\dim(\ker(P)\cap\ker(Q))\geq 1$. Thus $0$ is an eigenvalue of
$A$, and $b-d$ is an eigenvalue of $M-N$.

We compute the remaining two eigenvalues of $A$. Since the ranks of $P$ and $Q$ are 1, we have
$\operatorname{tr}(P)=\operatorname{tr}(Q)=1$, and therefore
$\operatorname{tr}(A)=\alpha-\beta$. Moreover,
\[
A^2=(\alpha Q-\beta P)^2
=\alpha^2Q+\beta^2P-\alpha\beta(QP+PQ).
\]
Taking traces and using $\operatorname{tr}(PQ)=\operatorname{tr}(QP)$, we get
\[
\operatorname{tr}(A^2)
=\alpha^2+\beta^2-2\alpha\beta\,\operatorname{tr}(PQ).
\]
Let the eigenvalues of $A$ be $0,\lambda_+,\lambda_-$. Then
\[
\lambda_++\lambda_-=\alpha-\beta,
\qquad
\lambda_+^2+\lambda_-^2
=\alpha^2+\beta^2-2\alpha\beta\,\operatorname{tr}(PQ).
\]
It follows that
$\lambda_+\lambda_-=
\alpha\beta(\operatorname{tr}(PQ)-1).$

Hence $\lambda_+,\lambda_-$ are the roots of
$\lambda^2-(\alpha-\beta)\lambda+\alpha\beta(\operatorname{tr}(PQ)-1)=0.$
Since $M-N=(b-d)I+A$, the eigenvalues of $M-N$ are
\[
b-d,\qquad b-d+\lambda_+,\qquad b-d+\lambda_-.
\]

Finally, $M-N$ has at most two distinct eigenvalues exactly when either one
of $\lambda_+,\lambda_-$ is equal to $0$, or when
$\lambda_+=\lambda_-$. These conditions are respectively equivalent to
\[
\alpha\beta(\operatorname{tr}(PQ)-1)=0
\qquad
\text{or}
\qquad
(\alpha-\beta)^2-4\alpha\beta(\operatorname{tr}(PQ)-1)=0.
\]
\end{proof}

\begin{remark}
By Corollary~\ref{cor:spectrum-difference}, one particularly natural way for
$M-N$ to have at most two distinct eigenvalues is that one of the two roots
above be equal to $0$, that is,
\[
\operatorname{tr}(PQ)=1.
\]

The condition $\operatorname{tr}(PQ)=1$ is not determined only by the image
lines of the rank-one projectors $P$ and $Q$. Indeed, if
\[
P=u\otimes\varphi,\qquad Q=v\otimes\psi,
\qquad \varphi(u)=\psi(v)=1,
\]
then
\[
\operatorname{tr}(PQ)=\varphi(v)\psi(u).
\]
Thus this condition also depends on the relative position of the kernels of
the two projectors.
\end{remark}

The next example shows that the most naive attempt to construct a finite non-cyclic group already fails.
\begin{example}[A first non-example]\label{ex:non-example}
The previous discussion suggests that, in order to construct a finite non-cyclic group in $\PP^5_{\mathbb C}$, one should try two matrices
with exactly two eigenvalues, having the same simple eigenspace but different double eigenspaces.

Let $\omega$ be a primitive third root of unity, and consider the rank-one idempotents
\[
Q=
\begin{pmatrix}
1&0&0\\
0&0&0\\
0&0&0
\end{pmatrix},
\qquad
P=
\begin{pmatrix}
1&1&0\\
0&0&0\\
0&0&0
\end{pmatrix}.
\]
Then $\operatorname{Im}(P)=\operatorname{Im}(Q)=\langle e_1\rangle$, but $\ker(P)\neq \ker(Q)$.

Now define
\[
M=\omega I+(\omega^2-\omega)Q
=
\begin{pmatrix}
\omega^2&0&0\\
0&\omega&0\\
0&0&\omega
\end{pmatrix},
\]
and
\[
N=\omega^2 I+(\omega-\omega^2)P
=
\begin{pmatrix}
\omega&\omega-\omega^2&0\\
0&\omega^2&0\\
0&0&\omega^2
\end{pmatrix}.
\]
Both matrices have exactly two eigenvalues. More precisely, $M$ has eigenvalues
$\omega^2,\omega,\omega$, while $N$ has eigenvalues $\omega,\omega^2,\omega^2$.

Moreover,
\[
M-N=
\begin{pmatrix}
\omega^2-\omega&\omega^2-\omega&0\\
0&\omega-\omega^2&0\\
0&0&\omega-\omega^2
\end{pmatrix}.
\]
Thus $M-N$ also has at most two distinct eigenvalues. In fact, after scaling
by $\omega^2-\omega$, it is projectively represented by
\[
\begin{pmatrix}
1&1&0\\
0&-1&0\\
0&0&-1
\end{pmatrix},
\]
which has square equal to the identity.

At first sight, this makes $(M,N)$ a plausible candidate for a finite
non-cyclic example. However, after rescaling projectively, one has
\[
[M]=
\left[
\begin{pmatrix}
\omega&0&0\\
0&1&0\\
0&0&1
\end{pmatrix}
\right],
\qquad
[N]=
\left[
\begin{pmatrix}
\omega^2&u&0\\
0&1&0\\
0&0&1
\end{pmatrix}
\right],
\]
where $u=\omega^2-1\neq 0$. Their product is
\[
[M][N]
=
\left[
\begin{pmatrix}
1&\omega u&0\\
0&1&0\\
0&0&1
\end{pmatrix}
\right].
\]
This is a nontrivial unipotent element of $\PGL_3(\mathbb C)$, hence it has
infinite order. Therefore the subgroup generated by $[M]$ and $[N]$ is
infinite.\end{example}

\begin{example}[Non-example with finite-order generators]\label{ex:non-example-fin-ord-gens}
    Let $\omega$, $P$, and $Q$ be as in the previous example. Define \[
    M=-\omega^2 I+(-\omega+\omega^2)Q=\begin{pmatrix}
        -\omega&0&0\\
        0&-\omega^2&0\\
        0&0&-\omega^2\\
    \end{pmatrix}
    \] and \[
    N=\omega^2 I+(\omega-\omega^2)P=\begin{pmatrix}
\omega&\omega-\omega^2&0\\
0&\omega^2&0\\
0&0&\omega^2
\end{pmatrix}.
    \] Then all the generators of the group $[M]$, $[N]$, $[M-I]$, $[N-I]$, and $[M-N]$ have finite order. The group $G_{\mathcal{L}}$ is still infinite, as one can detect by observing that the order of $[M-I][N]$ is infinite.
\end{example}

These observations suggest that, in characteristic $0$, the cyclic case may in fact exhaust all finite possibilities.

\begin{conjecture}\label{conj:finite-implies-cyclic}
Let $\mathcal L=\{L_\infty,L_0,L_1,\dots,L_r\}\subset \PP^{2n+1}_{\mathbb C}$, with $n>1$.
If $G_{\mathcal L}$ is finite, then $G_{\mathcal L}$ is cyclic.
\end{conjecture}

\section{The group of the groupoid in positive characteristic}\label{sec:charp}
We now turn to positive characteristic, where the rigidity of the characteristic-zero case breaks down.
Even finite cyclic groups may give rise to non-collinear orbits on the individual planes, and finite non-cyclic groups also occur.
The examples below show that the geometry in positive characteristic is genuinely different.
\subsection{Cyclic groups in positive characteristic}

The next example shows a behavior in sharp contrast with the characteristic-zero case.
Over $\mathbb C$, the finiteness of both $[M]$ and $[M-I]$ forces $M$ to have at most two distinct eigenvalues, hence a general orbit on each plane is collinear.
In characteristic $2$, this argument breaks down: we construct a matrix $M$ with finite projective order, and so does $[M+I]=[M-I]$, but $M$ has three distinct eigenvalues over $K$.
Therefore the action on each plane has three distinct eigenvalues, and a general orbit is no longer contained in a line.
\begin{example}[Characteristic $2$: four planes in $\PP^5_K$ defining the group $C_7$]
\label{ex:char2-2}
Assume $K$ is an algebraically closed field with $\operatorname{char}(K)=2$. Let
\[
M=\begin{pmatrix}
1&0&1\\
0&0&1\\
1&1&0
\end{pmatrix}
\in \GL_3(\mathbb F_2)\subset \GL_3(K).
\]
Let
\[
\mathcal L=\{L_0,L_\infty,L_I,L_M\}.
\]
A direct computation gives $M^7=I$ and $M+I=M^5$. Hence
\[
G_{\mathcal L}
=
\langle [M],[M+I]\rangle
=
\langle [M]\rangle
\subseteq \PGL_3(K).
\]
Therefore $G_{\mathcal L}\cong C_7$. In particular,
\[
|G_{\mathcal L}|=7
\]
and, for a point $p$ not on a transversal plane to $\mathcal L$, the orbit
$[p]_{\mathcal C_{\mathcal L}}$ has $28$ points.

We checked with Macaulay2, using random points to construct the orbit and then
projecting, that
\[
Z=[p]_{\mathcal C_{\mathcal L}}\subseteq \PP^5_K
\]
has Hilbert function
\[
h_Z=(1,5,15,7),
\qquad
h_{Z\cap L_0}=(1,2,3,1).
\]

On each plane $L\in\mathcal L$, the seven points of a general orbit slice
form a Fano plane configuration. Indeed, the minimal polynomial of $M$ is
$t^3+t^2+1$, hence
\[
M^3+M^2+I=0.
\]
Thus, for every vector $v$, the three points
\[
[v],\ [M^2v],\ [M^3v]
\]
are collinear, whenever they are distinct. 
Applying $M^k$
to the relation $M^3v=M^2v+v$ gives $M^{k+3}v=M^{k+2}v+M^kv$, so the triple $\{[M^k v], [M^{k+2} v], [M^{k+3} v]\}$ is collinear for every $k.$ Since the orbit has exactly seven points and the group has order 7, cycling through $k = 0, 1, \ldots, 6,$ $\mod 7$ gives exactly 7 collinear triples,  precisely the seven lines of the Fano plane. 
\end{example}
\begin{remark}
Example~\ref{ex:char2-2} is closely related to Ganger’s $\PP^5_K$ example \cite[Example 5.1]{ganger2024}: both arise from the Singer-cycle picture and lead to a cyclic group of order $7$ acting on planes in $\PP^5_K$. However, Ganger considers the canonical full spread and its transversals, while Example~\ref{ex:char2-2} isolates a $4$-plane subconfiguration and emphasizes the non-collinear orbit phenomenon in positive characteristic.
\end{remark}

\begin{remark}\label{rem:positive-char-difference}
Example~\ref{ex:char2-2} shows that the characteristic-zero picture does not
extend to positive characteristic.

Indeed, over $\mathbb C$, if
\[
\mathcal L=\{L_0,L_\infty,L_I,L_M\}
\]
and $G_{\mathcal L}$ is finite cyclic, then both $[M]$ and $[M-I]$ have
finite order in $\PGL_3(\mathbb C)$. This forces $M$ to have at most
two distinct eigenvalues. As a consequence, the orbit of a general point on
each plane is collinear.

In characteristic $2$, this argument breaks down. In
Example~\ref{ex:char2-2}, one has
\[
[M+I]=[M-I]=[M^5],
\]
so both $[M]$ and $[M-I]$ have finite order. However,
\[
\chi_M(t)=t^3+t^2+1
\]
is irreducible over $\mathbb F_2$. Hence $M$ has three distinct eigenvalues
over $\overline K$. Therefore the induced action on each plane is represented
by a matrix with three distinct eigenvalues, and a general orbit is not
contained in a line.

This explains why, in positive characteristic, finite cyclic groups may give
rise to non-collinear orbit slices on the individual planes. In particular, the
$h$-vector
\[
h_{Z\cap L_0}=(1,2,3,1)
\]
is consistent with the fact that the seven points of the orbit on
$L_0\simeq\PP^2_K$ are not collinear.
\end{remark}
Example~\ref{ex:char2-2} is the first instance of a general construction.
The following proposition shows that Singer cycles systematically produce
finite cyclic groups with non-collinear orbit slices in every positive
characteristic.

Let $q=p^e$, and view $\mathbb F_{q^3}$ as a $3$-dimensional vector space over $\mathbb F_q$.
Multiplication by a nonzero element $\alpha\in \mathbb F_{q^3}^*$ defines an $\mathbb F_q$-linear automorphism of
$\mathbb F_{q^3}$, hence an element of $\GL_3(\mathbb F_q)$ after a choice of basis.
Passing to projective classes, one obtains an embedding
\[
\mathbb F_{q^3}^*/\mathbb F_q^* \hookrightarrow \PGL_3(\mathbb F_q),
\]
whose image is a cyclic subgroup of order $q^2+q+1$.
Such a subgroup is called a \emph{Singer cycle}; see, for instance,
\cite[Chapter 2]{hirschfeld-thas}.

\begin{proposition}\label{prop:positive-char-singer}
Let $q=p^e$, let $K$ be an algebraically closed field of characteristic
$p$, and identify $\mathbb F_{q^3}$ with a $3$-dimensional vector space
over $\mathbb F_q$. Choose
$\alpha\in \mathbb F_{q^3}^*\setminus \mathbb F_q^*$
whose class generates the quotient group 
$\mathbb F_{q^3}^*/\mathbb F_q^*,$ 
which is cyclic of order $q^2+q+1$. Let
$M_\alpha\in \GL_3(\mathbb F_q)\subset \GL_3(K)$ be the matrix of the
$\mathbb F_q$-linear map $x\mapsto \alpha x$.

Let $\mathcal L=\{L_0,L_\infty,L_1,L_{M_\alpha}\}\subset \PP^5_K.$
 Then $L_0$, $L_\infty$, $L_1$, and $L_{M_\alpha}$ are pairwise skew, and
\[
G_{\mathcal L}
=
\langle [M_\alpha],[M_\alpha-I]\rangle
\cong C_{q^2+q+1}.
\]

Moreover, $M_\alpha$ has three distinct eigenvalues over $K$, namely
$\alpha, \alpha^q, \alpha^{q^2}.$ 
Hence, for a general point $p$ on any member of $\mathcal L$, the orbit of
$p$ on that plane is not collinear; indeed, it spans the whole plane.
\end{proposition}
\begin{proof}
Since $\alpha\neq 0$, the map $x\mapsto \alpha x$ is invertible, so
$M_\alpha\in\GL_3(\mathbb F_q)$. Also $\alpha\neq 1$, because
$\alpha\notin\mathbb F_q$. Hence $M_\alpha-I$ is invertible: it is the
matrix of the map $x\mapsto(\alpha-1)x$. Therefore $L_{M_\alpha}$ is skew
to $L_0$, $L_\infty$, and $L_1$, and the four planes are pairwise skew.

The projective class $[M_\alpha]$ depends only on the class of $\alpha$ in
$\mathbb F_{q^3}^*/\mathbb F_q^*.$

Likewise, $[M_\alpha-I]$ is the class of multiplication by $\alpha-1$, so it
also lies in the same cyclic subgroup
$\mathbb F_{q^3}^*/\mathbb F_q^*\subset \PGL_3(K).$
By assumption, the class of $\alpha$ generates this quotient. Hence
\[
\langle [M_\alpha],[M_\alpha-I]\rangle
=
\langle [M_\alpha]\rangle
\cong C_{q^2+q+1}.
\]

Finally, over $K$, the eigenvalues of $M_\alpha$ are the Frobenius
conjugates
$\alpha, \alpha^q, \alpha^{q^2}.$ 
They are distinct because $\alpha\notin \mathbb F_q$. Thus the induced
projective action on each plane has three distinct eigenspaces. Although Lemma~\ref{lem:general-cyclic-orbit-span} is stated over $\mathbb{C}$, its proof applies over any algebraically closed field whenever the lift  $A$ is diagonalizable, which holds here since $M_\alpha$ is semisimple with three distinct eigenvalues, therefore the orbit of a general point spans a
projective plane. In particular, it is not contained in a line.
\end{proof}

\subsection{Non-cyclic groups in positive characteristic}
We now pass from cyclic examples to a genuinely non-cyclic finite group in characteristic $2$.
\begin{example}[Characteristic $2$: a small finite example in $\PP^5_K$]\label{ex:char2}
Assume $K$ is an algebraically closed field with $\operatorname{char}(K)=2$. Let
\[
M=\begin{pmatrix}
1&0&1\\
0&0&1\\
1&1&0
\end{pmatrix},
\qquad
N=\begin{pmatrix}
1&1&1\\
1&0&1\\
0&1&1
\end{pmatrix}
\in \GL_3(\mathbb F_2)\subset \GL_3(K).
\]
Let
\[
\mathcal L=\{L_0,L_\infty,L_I,L_M,L_N\}\subset \PP^5_K.
\]

Since $\operatorname{char}(K)=2$, one has $M-I=M+I$, $N-I=N+I$, and $M-N=M+N$.
A direct computation gives
\[
M+I=\begin{pmatrix}
0&0&1\\
0&1&1\\
1&1&1
\end{pmatrix},\quad
N+I=\begin{pmatrix}
0&1&1\\
1&1&1\\
0&1&0
\end{pmatrix},\quad
M+N=\begin{pmatrix}
0&1&0\\
1&0&0\\
1&0&1
\end{pmatrix}.
\]
All five matrices $M$, $N$, $M+I$, $N+I$, and $M+N$ are invertible over $\mathbb F_2$, hence define elements of $\PGL_3(K)$.

A computation in Macaulay2 shows that the projective classes $[M]$ and $[N]$ generate
$\PGL_3(\mathbb F_2)$.
Since $\mathbb F_2^*=\{1\}$, one has
\[
\PGL_3(\mathbb F_2)=\GL_3(\mathbb F_2).
\]
Hence
\[
G_{\mathcal L}\cong \PGL_3(\mathbb F_2)\cong \GL_3(\mathbb F_2),
\]
so in particular $|G_{\mathcal L}|=168$.
Moreover,
\[
\GL_3(\mathbb F_2)\cong \PSL_2(\mathbb F_7).
\]
\end{example}

\begin{proposition}\label{prop:not-pgl2-char2}
In characteristic $2$, the group $G_{\mathcal L}$ of Example~\ref{ex:char2} does not embed in $\PGL_2(K)$.
\end{proposition}

\begin{proof}
Assume $\operatorname{char}(K)=2$ and suppose for contradiction that $G_{\mathcal L}$ embeds in $\PGL_2(K)$.
By Example~\ref{ex:char2}, one has
\[
G_{\mathcal L}\cong \PSL_2(\mathbb F_7),
\]
so in particular $|G_{\mathcal L}|=168$.

By Dickson's classification of finite subgroups of $\PGL_2$ in characteristic $p$, a finite subgroup of
$\PGL_2(K)$ whose order is divisible by $p$ is either
\begin{itemize}
\item $p$-semi-elementary (of affine type), hence has a nontrivial normal Sylow $p$-subgroup, or
\item isomorphic to $\PSL_2(\mathbb F_q)$ or $\PGL_2(\mathbb F_q)$ for some $q=p^f$.
\end{itemize}
See, for instance, \cite[Ch.~III, \S6]{HuppertEndlicheGruppen} or \cite[\S2]{SuzukiFiniteGroups}.

Now $\PSL_2(\mathbb F_7)$ has no normal Sylow $2$-subgroup, so it cannot be of affine type.
Hence it would have to be isomorphic to $\PSL_2(\mathbb F_{2^f})$ or $\PGL_2(\mathbb F_{2^f})$ for some $f\ge 1$.
Since $q=2^f$ is even, one has
\[
\PSL_2(\mathbb F_q)=\PGL_2(\mathbb F_q),
\]
and therefore
\[
|\PSL_2(\mathbb F_{2^f})|
=2^f(2^{2f}-1)
\in \{6,60,504,\dots\},
\]
which is never equal to $168$.
This contradiction shows that $G_{\mathcal L}$ does not embed in $\PGL_2(K)$.
\end{proof}

\medskip

\noindent{{\bf Acknowledgment.} Throughout the paper, the calculations were performed using the software Macaulay2~\cite{macaulay2}.   Favacchio is a member of GNSAGA-INdAM. The work of Favacchio was supported by the funding PREMIO\_SINGOLI\_RIC\_[2025] from the Department of Engineering, University of Palermo.}

\section*{}

\end{document}